\documentclass{amsart}

\usepackage{amssymb}

\usepackage{graphicx}
\usepackage{algorithm,algorithmic}
\usepackage[latin1]{inputenc}
\usepackage{amsfonts,amssymb,amsmath}
\usepackage{verbatim}
\usepackage{color}
\usepackage{algorithm,algorithmic}
\usepackage{enumerate}
\usepackage{url}

\newtheorem{maintheorem}{Main Theorem}
\newtheorem{theorem}{Theorem}[section]
\newtheorem{lemma}[theorem]{Lemma}
\newtheorem{proposition}[theorem]{Proposition}

\theoremstyle{definition}
\newtheorem{definition}[theorem]{Definition}

\theoremstyle{remark}
\newtheorem{remark}[theorem]{Remark}

\numberwithin{equation}{section}

\newcommand\Z{\ensuremath{\mathbb Z}}
\newcommand\Q{\ensuremath{\mathbb Q}}\newcommand\R{\ensuremath{\mathbb R}}

\newcommand\Nm{\operatorname{Nm}}

\newcommand\CF{\operatorname{CF}}

\def\cO{{\mathcal O}}

\begin{document}

\title[Continued Fractions in $2$-stage Euclidean Quadratic Fields]{Continued Fractions\\in $2$-stage Euclidean Quadratic Fields}


\author{Xavier Guitart}
\address{ Universitat Polit\`ecnica de Catalunya, Barcelona}
\curraddr{}
\email{xevi.guitart@gmail.com}

\author{Marc Masdeu}
\address{Columbia University, New York}
\curraddr{}
\email{masdeu@math.columbia.edu}

\thanks{Partially supported by Grants MTM2009-13060-C02-01 and 2009 SGR 1220.}

\subjclass[2010]{Primary 13F07, 11A55}

\date{\today}

\dedicatory{}

\begin{abstract}
We discuss continued fractions on real quadratic number fields of class number $1$. If the field has the
property of being $2$-stage euclidean, a generalization of the euclidean algorithm can be used to compute these
continued fractions. Although
it is conjectured that all real quadratic fields of class number $1$ are $2$-stage euclidean, this property has been
proven for only a few
of them. The main result of this paper is an algorithm that, given a real quadratic field of class number $1$, verifies
this conjecture, and produces as byproduct enough data to efficiently compute continued fraction expansions. If the
field was not $2$-stage euclidean, then the algorithm would not terminate.  As an application, we enlarge
the list of known $2$-stage euclidean fields, by proving that  all real quadratic  fields of class number
$1$ and discriminant less than $8000$ are $2$-stage euclidean.
\end{abstract}

\maketitle
\section{Introduction}
Let $F$ be a number field with maximal order $\cO_F$. Given a list
of elements $q_1,\dots,q_n\in \cO_F$ the (finite) \emph{continued
fraction} $[q_1,\dots,q_n]$ is the element of $F$ defined
inductively by:
$$[q_1]=q_1,\ \ [q_1,q_2]=q_1+\frac{1}{q_2},\ \ \dots,\  [q_1,q_2,\dots,q_n]=[q_1,[q_2,\dots,q_n]].$$
In the case $F=\Q$, it is a classical result that  every $x\in \Q$
is a continued fraction with coefficients in $\Z$, which can be
effectively computed by means of the euclidean algorithm. This property is no longer true for arbitrary $F$. Indeed,
one has the following
result (cf.~\cite{Co} Theorem 1, Corollary 3 and Proposition 13).
\begin{theorem}[Cooke-Vaser\v{s}tein] Suppose that $\cO_F^\times$ is
infinite. Then every element in $F$ is a continued fraction with
coefficients in $\cO_F$ if and only if $F$ has class number $1$.
\end{theorem}
Unlike  in the classical case, the proof of this theorem is not
constructive. Indeed, the ring $\cO_F$ is not euclidean in
general,  which makes it hard to effectively compute  continued
fractions of elements in $F$. There is a huge amount of literature devoted to euclidean rings, a number of
generalizations and their relation with continued fractions algorithms, mainly motivated by their applications to the
 arithmetic of number fields. A survey of these topics can be found in~\cite{Le}.

Continued fractions in number fields also arise in the computation of Hecke eigenvalues of automorphic
forms, via the modular symbols algorithm. See for instance~\cite{cre} for the case of imaginary quadratic fields,
or~\cite{DeVo} for
an account in the setting of Hilbert modular forms. As a more recent application, we mention that the computation of
continued fractions turns out to be a
critical step in the effective computation of ATR points in elliptic curves over real quadratic fields of class number
$1$. ATR points are Stark-Heegner points defined over Almost Totally Real fields. See~\cite[\S4]{DL} for a
discussion of the role played by continued fractions in this kind of computations, and~\cite[Example 6]{De} for another
example of their use.

In this note we restrict ourselves to the case of $F$ being a real
quadratic number field of class number $1$. We build on the approach taken
by Cooke in \cite{Co}, and we provide an algorithm that (under an appropriate version of the Generalized Riemann
Hypothesis) computes a continued fraction of any element $x\in F$. To
be more precise, we circumvent the problem that most of these
fields are not euclidean by exploiting the  property of being $2$-stage euclidean.

A \emph{$1$-stage division chain} for a pair of elements
$\alpha,\beta\in\cO_F$ is a pair $(q_1,r_1)$  of elements in $\cO_F$ satisfying $\alpha=q_1\beta+r_1$. A \emph{$2$-stage division chain}  is a
quadruple $(q_1,q_2,r_1,r_2)$ of elements in $\cO_F$ satisfying:
\begin{align*}
\alpha &= q_1\beta+r_1\\
\beta &= q_2r_1 +r_2.
\end{align*}
A field $F$ is said
to be \emph{$2$-stage
euclidean} (with respect to the norm $\Nm$ of $F/\Q$) if any pair of elements $\alpha,\beta\in \cO_F$, with $\beta\neq
0$, has a \emph{$k$-stage decreasing chain} for $k\leq 2$; that is, a $k$-stage division chain as above satisfying the
additional property
\[
|\Nm(r_k)|<|\Nm(\beta)|.
\]
In general, the notion of $n$-stage euclideanity is defined by means of division chains
of length at most $n$, but it is enough for our purposes to restrict to $n=2$.

Given a $2$-stage euclidean field $F$, (in fact, any $k$-stage euclidean field), a variation of the usual proof of the
fact that euclidean implies class number $1$ shows that $F$ is of class number $1$. Conversely, it is expected that all
real quadratic fields of class number $1$ are $2$-stage euclidean.
Indeed, in \cite{CW} it is proven that if certain Generalized Riemann Hypothesis holds then every real
quadratic field of class number $1$ is $2$-stage euclidean. In spite of this result, up to now only a few real
quadratic fields have been proven  to be $2$-stage euclidean: as reported in  \cite[p. 14]{Le}, they are the fields
$\Q(\sqrt{m})$ with $m$ belonging to the set
\begin{small}
\begin{align}\label{eq: known real quadratic 2-stage}
&  \{{\bf 2},{\bf 3},{\bf 5},{\bf 6},{\bf 7},{\bf 11},{\bf 13},14,{\bf 17},{\bf 19},{\bf 21}, 22, 23,{\bf 29}, 31,{\bf
33},{\bf 37}, 38,{\bf 41}, 43,\\\nonumber
& 46,47, 53,{\bf 57}, 59, 61, 62,
67, 69, 71,{\bf 73}, 77, 89, 93, 97, 101, 109, 113,129,\\\nonumber
  & 133, 137, 149, 157, 161, 173, 177, 181, 193, 197, 201, 213, 253\},\nonumber
\end{align}
\end{small}
where those numbers appearing in bold face correspond to the complete list of norm-euclidean rings. The  purpose of this note is to present an algorithm for checking  $2$-stage euclideanity of real quadratic fields,
thus allowing the computation of continued fractions. The main result of the article is the following theorem.
\begin{maintheorem}\label{main theorem}
There exists an algorithm that:
\begin{enumerate}[(i)]
 \item\label{mainthm1} accepts as input a real quadratic field $F$; if $F$ is $2$-stage euclidean the algorithm terminates and it proves
that $F$ is $2$-stage euclidean.
\item\label{mainthm2}  if $F$ is $2$-stage euclidean, after finishing step $(i)$ it accepts as input any $x\in F$ and it computes a
continued fraction with
coefficients in $\cO_F$ for $x$.
\end{enumerate}
\end{maintheorem}

\begin{remark}
If the field $F$ is not $2$-stage euclidean then the algorithm will not terminate. However, in an actual implementation
the algorithm  would either run out of memory or break due to rounding errors. However, as expected, we haven't been
able to observe this phenomenon because all tested fields are indeed $2$-stage euclidean.
\end{remark}

As will be explained in more detail in the subsequent sections, for a
given field $F$ what  step (\ref{mainthm1}) does is a
precomputation of the data that is needed in order to compute a $2$-stage decreasing chain for any $\alpha,\beta\in
\cO_F$ with
$\beta\neq 0$. If
$F$ is $2$-stage euclidean the algorithm succeeds in this precomputation, which  at the same time constitutes  a proof
that
$F$ is $2$-stage euclidean. Then the data generated in (\ref{mainthm1}) is used  in (\ref{mainthm2}) to
compute the continued fraction of any $x\in F$.

An implementation of the algorithm has been submitted as a patch to  Sage, and it  is available at
\url{http://trac.sagemath.org/sage_trac/ticket/11380}. As an application, we have
extended  list \eqref{eq: known real quadratic 2-stage} by proving that all real quadratic fields of class
number $1$ and discriminant up
to $8000$ are $2$-stage euclidean.

The plan of the paper is as follows: In Section \ref{section: Continued fractions and 2-stage euclidean fields} we
recall the basic notions of $2$-stage euclidean fields and their relation with the computation of continued fractions.
In Section \ref{section: The algorithm} we describe and prove the correctness of the algorithm in Main Theorem \ref{main
theorem}.  In Section \ref{section:tables} we comment on some implementation details and we also include some data arising from numerical
experiments. As we will see, they suggest a measure of euclideanity for quadratic fields that, to the best of our
knowledge, have not been considered before.

It is a pleasure to thank Jordi Quer for his comments on an earlier version of the
 manuscript. We are also grateful to the referee for valuable observations and suggestions.

\section{Continued fractions and $2$-stage euclidean fields}\label{section: Continued fractions and 2-stage euclidean
fields}
We begin this section by recalling the basic definitions and properties  that we will use. Let $m$ be a positive
squarefree integer and let  $F=\Q(\sqrt m)$. Let $\omega=(1+\sqrt m)/2$ if $m\equiv 1 \mod 4$ and
 $w=\sqrt{m}$ if $m\equiv 2,3 \mod 4$, so that the ring of integers is $\cO_F=\Z +\Z\omega$. Let $\alpha/\beta$ be an
element of $F$. If $F$ is $2$-stage euclidean one can find a
$k$-stage decreasing chain for the pair $\alpha,\beta$ with $k\leq 2$. If the last residue $r_k$ is not zero, one can then
repeat this process to end with a division chain
\begin{align*}
\alpha &= q_1\beta+r_1\\
\beta &= q_2r_1 +r_2,\\
r_1&=q_3r_2+r_3\\
r_2&=q_4r_3+r_4\\
&\vdots\\
r_{n-2}&=q_nr_{n-1}+r_n
\end{align*}
with $r_n=0$, because the norm of the corresponding residue decreases in absolute value at worst every two steps. The 
classical
formulas for the
case of rational numbers (cf. \cite[\S 10.6]{HW}) show that  $\alpha/\beta$ is then equal to the continued fraction
$[q_1,\dots,q_n]$. Therefore, part (\ref{mainthm2}) of Main Theorem \ref{main theorem} is straightforward if one can compute
$2$-stage decreasing chains for arbitrary pairs of elements in $\cO_F$. We remark, however, that $\alpha/\beta$ may
admit many different
representations as a
continued fraction, since $2$-stage decreasing chains are not unique.

The following is a particular case of \cite[Corollary 1]{Co}.
\begin{proposition}\label{proposition: condition of existence of a 2-chain}
 There exists a $2$-stage decreasing chain for $\alpha,\beta$  if and only if there exists a
continued
fraction of length $2$, say $[q_1,q_2]$, such that
$$\left |\Nm\left(\frac{\alpha}{\beta}-[q_1,q_2])\right)\right |<\frac{1}{|\Nm(q_2)|}.$$
\end{proposition}
Let $v_1,v_2$ denote the two embeddings of $F$ into $\R$, and let  $v=(v_1,v_2)\colon F\hookrightarrow\R^2$. Concretely,
it is given by $1\mapsto (1,1)$ and $\sqrt{m}\mapsto
(\sqrt{m},-\sqrt{m})$.
We will often identify $F$ with $v(F)$, even without explicitly mentioning $v$. The norm of $F$
extends to $\R^2$ via the formula $$\Nm(x,y)=xy,\ \ \text{for}\  x,y\in \R.$$

Let $\CF_2$ denote the set of continued fractions of length at most $2$. Any element in $\cO_F$ can be expressed
as a continued fraction of length $2$, so $\CF_2$ is also the set of continued fractions of length exactly $2$. For a
positive integer $n$, let
\[
 \CF_2(n)=\{q=[q_1,q_2]\in \CF_2\colon |\Nm(q_2)|\leq n\}.
\]
For $q=[q_1,q_2]\in \CF_2$ define
\begin{equation*}
 V(q)=\left\{x\in \R^2 \ \colon \ |\Nm(x-q)| <\frac{1}{| \Nm(q_2)| }\right\}.
\end{equation*}
In the subsequent sections it will be useful to refer to $q_2$ as the \emph{denominator of $V(q)$}. The region $V(q)$ is
bounded by the hyperbolas
\begin{equation*}
 (x-x_0)(y-y_0)= \frac{\pm 1}{|\Nm (q_2)|},
\end{equation*}
where $(x_0,y_0)=v(q)$. From Proposition \ref{proposition: condition of existence of a 2-chain}   we see that $F$ is
$2$-stage euclidean if and only if $\R^2$ can be covered by open sets of the form   $V(q)$, with $q\in \CF_2$. The
knowledge of such a
covering also translates into a method for computing a $2$-stage decreasing chain for a pair $\alpha,\beta$: if
$\alpha/\beta$
belongs to $V([q_1,q_2])$, then $q_1$ and $q_2$ are the quotients of such a chain.

Let $\gamma$ be an element in $\cO_F$. Then $x$ belongs to $ V([q_1,q_2])$ if and only if $(x-\gamma)$ belongs to $
V([q_1-\gamma,q_2])$. So
instead of  $x\in F$ one can work with $\overline x$, its class modulo $\cO_F$, which as an element of $\R^2$ lies in
 the fundamental domain
\[
 D=\{av(1)+bv(w)\ \colon \ a,b\in [0,1)\}.
\]
The advantage is that $\overline D$ is compact, so finite coverings are enough.
\begin{proposition}\label{proposition: finitely many sets}
The quadratic field $F$ is $2$-stage euclidean if and only if $D$ can be covered by finitely many hyperbolic regions
$V(q)$ with $q$
belonging to $\CF_2$.
\end{proposition}
From Proposition \ref{proposition: finitely many sets} we can already see the idea of an algorithm for checking whether
$F$ is $2$-stage euclidean. It is easy to define  an ordering on the set of continued fractions
$q=[q_1,q_2]\in \CF_2$. One can then generate such $q$'s in order,  and check at each step whether the sets
$V(q)$ for the $q$ generated so far already cover $D$. If $F$ is $2$-stage euclidean this process will necessarily
finish,
producing a
finite list of $V(q)$'s
that cover $D$. One can then compute  a $2$-stage decreasing chain
for a pair $\alpha,\beta$ by finding a $V(q)$ that contains $\overline{\alpha/\beta}$.

 However, working with all the sets $V(q)$ for $q\in \CF_2$ is readily seen to be computationally unfeasible.
Therefore, one wants to work only with a few of the possible centers, but in a way that the algorithm is still
guaranteed to finish. This is essentially what our algorithm does. At this point we remark that the algorithm presents
two critical points:
\begin{enumerate}
 \item How to choose the centers $q$ for the regions $V(q)$ to be considered.
  \item How to check, algorithmically, whether  a collection of sets $V(q)$  covers $D$.
\end{enumerate}
The next section is devoted to discuss in detail the algorithm and the implementation of these two steps.

\section{The algorithm}\label{section: The algorithm}
In this section we address the two main points raised at the end of the previous section. The centers  that will be
considered come from the observation that, for each positive integer $N$, there are  finitely many elements
$q=[q_1,q_2]$ inside the fundamental domain $D$ with $|\Nm(q_2)|\leq N$. We will take small translates of these centers
by elements of $\cO_F$, which moves them outside $D$, but as long as the corresponding regions still
intersect $D$. The following definitions and results make this more precise.

Given a positive integer $N$, denote by $Q_{N}$ the set
consisting of continued fractions $q=[q_1,q_2]$ of length  two with $|\Nm(q_2)|\leq N$ and such that $q$
belongs to $D$:
\[
Q_{N}=\{q=[q_1,q_2] \ \colon |\Nm(q_2)|\leq N\text{ and } q\in D\}.
\]
For a positive integer $T$, we also define the following set of translates of elements in $Q_N$:
\[
Q_{T,N}=\{q=a+b\omega \in Q_N+\cO_F \ \colon |b|< T\text{ and } V(q)\cap  D\neq\emptyset\}.
\]
\begin{proposition}\label{Qs are finite and computable}
The sets $Q_N$ and $Q_{T,N}$ are finite and effectively computable.
\end{proposition}
\begin{proof}
First we consider the set $Q_N$. There is a finite number of ideals of norm bounded by $N$,
and there are algorithms to compute them. Since $\cO_F$ is a principal ideal domain, the set of ideals of norm
 up to $N$ is of the form
\[
\{(\alpha_1),\ldots,(\alpha_k)\},
\]
for some (non-canonical) choice of representatives $\alpha_i$. If $\beta$ is any element of norm less than or equal to $
N$, it must be of the
form $\beta=u\alpha_i$ for some $i$  and some unit $u$. Therefore:
\[
\frac 1\beta=\frac{1}{u\alpha_i}=\frac{u^{-1}}{\alpha_i}\in \frac{1}{\alpha_ i}\cO_F.
\]
Since we are looking for representatives modulo the action of the additive group $\cO_F$, all  elements of $Q_{N}$ are
to be
found in
\[
 \bigcup_{i=1}^k \left(\cO_F+\frac{1}{\alpha_i}\cO_F\right)/\cO_F\simeq \bigcup_{i=1}^k \cO_F/\alpha_i\cO_F,
\]
which is finite and computable.

Given positive integers $T$ and $N$, and an element $q=[q_1,q_2]\in Q_{N}$, define the set $Q_{q,T,N}$ as:
\[
Q_{q,T,N}=\{a+b\omega\in q+\cO_F\colon |b|<T,\,V(a+b\omega)\cap D\neq \emptyset\}.
\]
To prove the finiteness of $Q_{T,N}$ it is enough to show that the sets $Q_{q,T,N}$ are finite for each $q\in Q_N$.
Write $q=r+s\omega$ for some $r,s\in\Q$. If $q'$ is an element of $Q_{q,T,N}$, one can write $q'=r+d+(s+t)\omega$ where
$d$ and $t$ belong to $\Z$.
Since the absolute value of $s+t$ needs to be bounded by $T$ and $s$ is fixed, there is a finite number of possible
choices for $t$.
It remains to show  that for each value of $t$ there are finitely many possibilities for $d$. Let
$v(r+(s+t)\omega)=(x_0,y_0)$. Then $v(q')=v(r+d+(s+t)\omega)=(x_0+d,y_0+d)$. The
hyperbolic region $V(q')$ is contained in the union of two  strips in $\R^2$:
\[
V(q')\subset R_{d}^x\cup R_{d}^y,
\]
where
\[
R_{d}^x=\{(x,y)\in\R^2\colon |(x-x_0-d)|<|\Nm(q_2)|^{-1/2}\},
\]
\[
R_{d}^y=\{(x,y)\in\R^2\colon |(y-y_0-d)|<|\Nm(q_2)|^{-1/2}\},
\]
Since $x_0$ and $y_0$ are fixed, it is clear that $R_d^x\cup R_d^y$ intersects $D$ for finitely many values of $d$.
\end{proof}

The following lemma allows us to prove that a certain region of $\R^2$ is covered by hyperbolic regions by doing a
finite amount of computation. Its proof is elementary and follows easily from the shape of the regions $V(q)$.
\begin{lemma}
\label{lemma:box}
  Let $R$ be a box in $\R^2$ of the form:
\[
R=R(x_0,x_1,y_0,y_1)=\{(x,y)\in \R^2 \colon x_0\leq x\leq x_1 \text{ and } y_0\leq y\leq y_1\}.
\]
Then $R$ is  contained in $V(q)$ if each of its four corners is.
\end{lemma}
\begin{proof}
Doing a translation in $\R^2$ we may and do assume that $V(q)$ is of the form:
\[
V(q)=\left\{(x,y)\in \R^2~\colon~ |xy|<\epsilon\right\},
\]
for some positive $\epsilon$. If the four corners of $R$ are contained in $V(q)$, then:
\[
|x_iy_j|<\epsilon,\quad i,j\in\{0,1\}.
\]
Given $(x,y)$ belonging to $R$, we have that:
\[
|x|\leq \max\{|x_0|,|x_1|\},
\]
and that
\[
|y|\leq \max\{|y_0|,|y_1|\}.
\]
Therefore:
\[
|xy|=|x||y|\leq \max\{|x_0|,|x_1|\}\max\{|y_0|,|y_1|\}<\epsilon.
\]
\end{proof}

Let $R_0$ be a box as in Lemma~\ref{lemma:box} such that $R_0$ contains the fundamental domain $D$ of $F$. For each
positive integer $n$, subdivide $R_0$ into $4^n$ identical boxes, and let $S(R_0)_n$ be the set of these. The following
result plays a crucial role in showing the correctness of the algorithm.
\begin{lemma}
\label{lemma:2stage-criterion}
The field $F$ is $2$-stage euclidean if and only if there exists a finite set $Z$ of hyperbolic regions and a positive
integer $L$ such that each box in $S(R_0)_L$ is contained in at least one of the regions of $Z$.
\end{lemma}
\begin{proof}
First assume that $F$ is $2$-stage euclidean. Therefore the box $R_0$ can be covered by hyperbolic regions $V(q)$. Since
these are
open and $R_0$ is compact, there is a finite set of hyperbolic regions $Z$ such that $R_0$ is covered by  regions in
$Z$.

Let $S(R_0) =\cup_{n=1}^\infty S(R_0)_n$ and let $U\subset R_0$ be the complement of the set of
corners:
\[
U=R_0\setminus \bigcup_{R=R(x_0,x_1,y_0,y_1)\in S(R_0)}\{(x_0,y_0),(x_0,x_1),(x_1,y_0),(x_1,y_1)\}.
\]
Note that $U$ is dense in $R_0$. For each point $x$ in $U$, let $V_x$ be a region in $Z$ containing $x$, and let $R_x$
be an element of $S(R_0)$ which is contained in $V_x$ and such that $x$ belongs to $R_x$. Let $n_x$ be the least integer
such that $R_x$ belongs to $ S(R_0)_n$.

The set $R_0$ is covered by the interiors of the boxes $R_x$ as $x$ varies in $R_0$, and therefore we can extract a
finite covering, say:
\[
R_0=\bigcup_{i=1}^l R_{x_i}.
\]
Set $L$ to be the maximum of the integers $n_{x_1},\ldots,n_{x_l}$. It is easily verified that the set $Z$ and the
integer $L$ satisfy the condition of the lemma.

The converse is obviously true, since the hyperbolic regions belonging to $Z$ already cover $R_0$, which contains $D$.
\end{proof}

The algorithm performing Part (\ref{mainthm1}) of Main Theorem \ref{main theorem} can easily be described in recursive form. Algorithm~\ref{alg:recursive} is the recursive function \textbf{SOLVE}, which accepts as input a box $R$. Algorithm~\ref{alg:main} is the main loop.
 \begin{algorithm}[H]
 \caption{\textbf{SOLVE}}
\label{alg:recursive}
\begin{algorithmic}
  \REQUIRE A box $R$. Global $T,N$ and a computed $Q_{T,N}$. A global vector $Z$ of regions used so far.
  \ENSURE The box $R$ is covered by regions in $Z$.
  \IF{there is $q\in Q_{T,N}$ such that $R\subset V(q)$}
    \STATE Append $q$ to $Z$.
  \ELSE
    \STATE Increase $T$ and $N$ and re-compute $Q_{T,N}$.
    \STATE Subdivide $R$ into four equal boxes $R_1,R_2,R_3,R_4$.
    \FOR{i=1 \TO 4}
      \STATE \textbf{SOLVE}$(R_i)$
    \ENDFOR
  \ENDIF
  \RETURN
\end{algorithmic}
\end{algorithm}
\begin{algorithm}[H]
\caption{Proves that $F$ is $2$-stage euclidean.}
\label{alg:main}
\begin{algorithmic}
\REQUIRE $F$ a real quadratic number field.
\ENSURE $F$ is $2$-stage euclidean.
\STATE Find a box $R_0$ in $\R^2$ that contains the fundamental domain $D$.
\STATE $Z\leftarrow []$.
\STATE Fix some initial $N$ and $T$ and compute the set $Q_{T,N}$.
\STATE \textbf{SOLVE}$(R_0)$
\RETURN $Z$, a list of regions covering the fundamental domain $D$.
\end{algorithmic}
\end{algorithm}
\begin{theorem}
 Algorithm~\ref{alg:main} terminates if and only if $F$ is $2$-stage euclidean.
\end{theorem}
\begin{proof}
Note first that Algorithm~\ref{alg:main} terminates if and only if the function call to \textbf{SOLVE}$(R_0)$
terminates. Suppose
that $F$ is $2$-stage euclidean, and let $Z$ and $L$ be as given by Lemma~\ref{lemma:2stage-criterion} applied to the
box $R_0$. Let $T,N$ be such that
\[
Z\subseteq Q_{T,N}.
\]
The size of the boxes passed to the \textbf{SOLVE} function is divided by four each time that the recursion depth
increases. On the other hand, both $T$ and $N$ increase as well with the recursion depth. Therefore, for a sufficiently
large recursion depth we will have $T$ and $N$ satisfying the above containment, and at the same time boxes considered
will belong to $S(R_0)_{L'}$ for some $L'>L$. Hence the algorithm terminates in finite time.

Conversely, if the algorithm terminates it exhibits a list of regions that covers the fundamental domain. Therefore $F$
must be $2$-stage euclidean.
\end{proof}

\section{Numerical experiments and a measure of euclideanity}\label{section:tables}

As an application of the algorithm we have verified the following result.
\begin{theorem}\label{th: up to 8000}
All real quadratic number fields of class number $1$ and discriminant  less than $8,000$ are $2$-stage euclidean.
\end{theorem}
The computations were performed using Sage in a laptop with processor Intel Core 2 Duo T7300 / 2.0 GHz  and 2.0 GB of
RAM. The time needed for checking the $2$-stage euclideanity of a given number field tends
to grow with the discriminant.
For instance, for discriminants of size about $100$ it takes no more than a few seconds, while for
discriminants of size about $8000$ it can take up to several hours.

In spite of this, the size of the discriminant is not the only factor that determines the computational cost. For
instance, we have observed that discriminants of similar size can lead to very different times of execution, depending
 on the number of small primes that are inert in the field. Recall that the algorithm terminates when it covers the
domain $D$ with sets $V(q)$. These $V(q)$ are bounded by hyperbolas of the type  $(x-x_0)(y-y_0)=\pm 1/|\Nm(q_2)|$,
where $q_2$ is
the denominator of $q$. A prime $p$ that
is not inert is $F$ leads to hyperbolas of the type $(x-x_0)(y-y_0)=\pm 1/p$. However, if $p$ is inert it  leads to
hyperbolas of the type $(x-x_0)(y-y_0)=\pm 1/p^2$ instead, which are likely to cover a smaller part of $D$. As an
illustration of this phenomenon, we mention that it took $68$ seconds to check the $2$-stage euclideanity of
$\Q(\sqrt{1273})$, where  primes $2$, $3$, and $5$ are not inert, whereas it took $1131$ seconds to check
$\Q(\sqrt{1253})$, where the only prime less than $20$ that is not inert is $7$.

The sizes  of the   hyperbolic regions are also related to another  factor that can substantially
affect the running time: the maximum norm of the denominators  of the regions $V(q)$ that are needed to cover $D$. In
order to
analyze  this influence, it is useful to make the following definition.
\begin{definition}
Let $F$ be a $2$-stage euclidean field with fundamental domain $D$. We say that
$F$ is \emph{$n$-smooth euclidean} if
\[
 D\subseteq \bigcup_{q\in  \CF_2(n)} V(q);
\]
that is,  if $D$ can be covered by using regions of denominator up to $n$.
\end{definition}
Let $n$ be the smallest integer such that $F$ is $n$-smooth euclidean. The complexity of the algorithm depends on $n$,
because it determines the number of hyperbolic regions with center in $D$ to be computed. Indeed, the number of ideals
of
norm $m$ is $O(m^\varepsilon)$. By (the proof) of Proposition \ref{Qs are finite and computable},  there are at
most $m$ centers of hyperbolic regions lying in $D$ for each ideal of norm $m$, thus giving $O(m^{1+\varepsilon})$
centers in $D$ whose denominator has norm $m$. Since one has to consider all ideals of norm $m\leq n$, the cardinality
of the set $Q_{n}$ is at most $O(n^{2+\varepsilon})$.

However, the algorithm actually works with translations of elements in $Q_n$. That is, the centers to be considered lie
in $Q_{T,n}$ for some $T$, which corresponds to the maximum length of translations. Unfortunately, to fully determine
the complexity of the algorithm we lack an estimation of  $T$, as well as of the integer $L$ in Lemma
\ref{lemma:2stage-criterion}, which gives the number of boxes that are to be checked by the function {\bf SOLVE}. We
remark that in the implementation used to prove Theorem \ref{th: up to 8000} a value of $T=5$ was enough to solve for
all
discriminants.

The value of the smallest $n$ such that $F$ is $n$-smooth euclidean  does not only depend  on the size of the
discriminant,
but also on the splitting behavior of small primes. As an illustration of this fact, Figure~\ref{fig} plots  the maximum
norm of the denominators that our implementation used to cover the tested number fields,
according to their discriminant.

\begin{figure}[h]
\includegraphics[width=80mm]{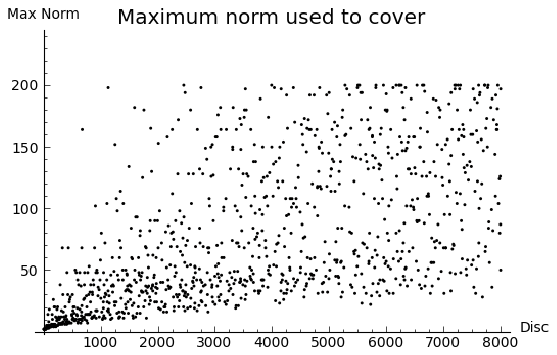}
\caption{}\label{fig}
\end{figure}

When carrying out this test we initialized the maximum norm to $N=200$. This explains why the points tend to accumulate
towards this value as the discriminant increases. Also notice how, even if the points appear distributed in a random
fashion, there is a region on the lower part of the graph where no point lies. This region increases with the
discriminant, and in the next subsection we will give an explanation to this phenomenon.

Figure \ref{fig2} is similar to Figure \ref{fig}, but only the data for some of the number fields are plotted: those in
which both $2$ and $3$ remain inert are shown as red squares, whereas those in which $2$ and $3$ split are shown as blue
circles.

\begin{figure}[h]
\includegraphics[width=80mm]{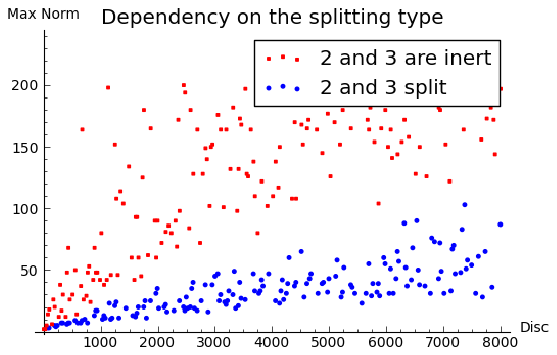}
\caption{}\label{fig2}
\end{figure}

We remark that if the algorithm manages to cover $D$ using denominators of norm up to $n$, this implies that $F$
is $n$-smooth euclidean. However, this does not rule out the possibility that the field could be $m$-smooth euclidean
for some
$m<n$. If we knew a priori the smallest $n$ such that $F$ is $n$-smooth euclidean, we could initialize  $N$ in Algorithm
\ref{alg:main}
 to this value and then only increasing $T$ the algorithm would cover $D$. However, the
value of the smallest $n$ is not known a priori, so that  the algorithm is not guaranteed to
finish if one only increases $T$. Therefore, the algorithm may have to eventually increase also the
maximum norm $N$, and this can lead to considering norms higher than what is strictly necessary.
Actually, the way in which $T$ and $N$ are increased turns out to be the most critical implementation parameter of the
algorithm, as for the running time is concerned. In the implementation we used for proving Theorem \ref{th: up to 8000}
we took a constant value of $T=5$,  and the increasing step for $N$ was proportional to the part of $D$ not being
covered (the proportional constant found by fine-tuning).

\subsection*{A measure of euclideanity} Besides its computational influence in our implementation, the smallest $n$
such that $F$ is $n$-smooth euclidean can also be
interpreted as a measure of how far is $F$ from being euclidean. Indeed, it is easily seen that  $F$ is
$n$-smooth euclidean
 if and only if for any pair $\alpha,\beta\in \cO_F$, $\beta \neq 0$, one can find $q,r\in \frac{1}{n}\cO_F$ such that
\[
 \alpha=q\beta +r,
\]
with $|\Nm(r)|<|\Nm(\beta)|$. In particular, $F$ is $1$-smooth euclidean if and only if it is euclidean. In this way,
the
following statement can be seen as a generalization of the classical result on the existence of finitely many euclidean
real quadratic fields.

\begin{theorem}\label{th: finitely many n-coverable}
 Let $n$ be a positive integer. There exist only finitely many $n$-smooth euclidean real quadratic  fields.
\end{theorem}
This is a consequence of a well known property of euclidean minima of real quadratic fields. We recall that the
\emph{euclidean minimum of $F$} is defined to be
\[
M(F)=\inf \{\mu\in \R \colon \forall x\in F\  \exists \ y\in  \cO_F \text{ such that } |\Nm(x-y)|\leq \mu
\}.
\]
Denote by $d(F)$ the discriminant of $F$. By a result of Ennola \cite{En} we have that for real quadratic fields
\[
 M(F)\geq \frac{\sqrt{d(F)}}{16+6\sqrt{6}}.
\]
Theorem \ref{th: finitely many n-coverable}  follows immediately from the following lemma.
\begin{lemma}\label{lemma: theoretical bound}
Let $n$ be a positive integer and let $t=\operatorname{lcm}(1,2,3,\dots,n).$  If   $d(F)>(16+6\sqrt{6})^2\cdot t^4$
then there exists an element $z\in F$ such that
\[
\left|\Nm(z-q) \right|> 1\ \ \ \ \text{for
all $q\in \CF_2(n)$}.
\]
\end{lemma}
\begin{proof}
Let $z_0 $ be an element in $F$ with the property that
\[
 |\Nm(z_0-\alpha)|\geq\frac{\sqrt{d(F)}}{16+6\sqrt{6}}\ \  \text{ for all $\alpha\in \cO_F$},
\]
 and let $z=\frac{z_0}{t}$. Then
\[
 |\Nm(z-\frac{\alpha}{t})|\geq \frac{\sqrt{d(F)}}{t^2\cdot (16+6\sqrt{6})}>1\ \  \text{ for all $\alpha\in \cO_F$}.
\]
This finishes the proof, because  $\CF_2(n)$ is contained in $ \frac{1}{t}\cO_F$.
\end{proof}


\begin{thebibliography}{99}
\bibitem{Co}
G. E. Cooke, {\it A weakening of the Euclidean property for integral domains and applications to algebraic number
theory I.} J. Reine Angew. Math. 282 (1976), 133--156.
\bibitem{Co2}
G. E. Cooke, {\it A weakening of the Euclidean property for integral domains and applications to algebraic number
theory. II.} J. Reine Angew. Math. 283/284 (1976), 71--85.
\bibitem{CW}
G. E. Cooke, P. Weinberger, {\it On the construction of division chains in algebraic number rings, with applications to
{${\rm SL}_{2}$}.} Comm. Algebra. vol 3 (1975), 481--524.
\bibitem{cre}
J. E. Cremona, {\it Hyperbolic tessellations, modular symbols, and elliptic curves over complex quadratic fields}.
Compositio Mathematica, 51 no. 3 (1984), p. 275--324
\bibitem{DL}
H. Darmon and A. Logan, {\it Periods of Hilbert modular forms and rational points on elliptic curves.} Int. Math. Res.
Not. 2003, no. 40, 2153--2180.
\bibitem{Da}
H. Davenport, {\it Indefinite binary quadratic forms, and Euclid's algorithm in real quadratic fields.} Proc.
London Math. Soc. (2) 53 (1951), 65--82.
\bibitem{De}
L. Demb\'el\'e, {\it An algorithm for modular elliptic curves over real quadratic fields.}
Experiment. Math. 17 (2008), no. 4, 427--438.
\bibitem{DeVo}
L. Demb\'el\'e, J. Voight, {\it Explicit methods for Hilbert modular forms.} To appear in ``Elliptic Curves, Hilbert
modular forms and Galois deformations''. {\tt arXiv:1010.5727v2.}
\bibitem{En}
V. Ennola, {\it On the first inhomogeneous minimum of indefinite binary quadratic forms and
Euclid's algorithm in real quadratic fields}
Ann. Univ. Turku. Ser. A I 28 (1958), 58pp.
\bibitem{HW}
G. H. Hardy, E. M.  Wright, {\it An introduction to the theory of numbers.} Sixth edition. Revised by D. R.
Heath-Brown and J. H. Silverman. With a foreword by Andrew Wiles. Oxford University Press, Oxford, 2008. xxii+621 pp.
ISBN: 978-0-19-921986-5, 11-01
\bibitem{Le}
F. Lemmermeyer, {\it The Euclidean algorithm in algebraic number fields.} Exposition. Math. 13 (1995), no. 5, 385--416.
\end{thebibliography}
\end{document}